\documentclass[10pt, reqno]{amsart}
\usepackage{amssymb,latexsym,amsmath,amsthm}

\newtheorem*{thm*}{Theorem 1}

\newtheorem{thm}{Theorem}

\newtheorem{exam}{Example}

\newtheorem{lemma}{Lemma}

\newtheorem{remark}{Remark}

\newtheorem{prop}{Proposition}

\begin{document}

\title{title}

\author{A. Aghajani}
\address{School of Mathematics, Iran University of Science and Technology, Narmak, Tehran, Iran \&  School of Mathematics, Institute for Research in Fundamental Sciences (IPM), P.O. Box 19395-5746, Tehran, Iran}
\email{aghajani@iust.ac.ir}

\author{C. Cowan}
\address{Department of Mathematics, University of Manitoba, Winnipeg, Manitoba, Canada R3T 2N2}
\email{craig.cowan@umanitoba.ca}

\author{S.H. Lui}
\address{Department of Mathematics, University of Manitoba, Winnipeg, Manitoba, Canada R3T 2N2}
\email{shaun.lui@umanitoba.ca}

\def\d{ \partial_{x_j} }
\def\Na{{\mathbb{N}}}

\def\Z{{\mathbb{Z}}}

\def\IR{{\mathbb{R}}}

\newcommand{\E}[0]{ \varepsilon}

\newcommand{\la}[0]{ \lambda}

\newcommand{\s}[0]{ \mathcal{S}}

\newcommand{\AO}[1]{\| #1 \| }

\newcommand{\BO}[2]{ \left( #1 , #2 \right) }

\newcommand{\CO}[2]{ \left\langle #1 , #2 \right\rangle}

\newcommand{\R}[0]{ \IR\cup \{\infty \} }

\newcommand{\co}[1]{ #1^{\prime}}

\newcommand{\p}[0]{ p^{\prime}}

\newcommand{\m}[1]{   \mathcal{ #1 }}

\newcommand{ \W}[0]{ \mathcal{W}}

\newcommand{ \A}[1]{ \left\| #1 \right\|_H }

\newcommand{\B}[2]{ \left( #1 , #2 \right)_H }

\newcommand{\C}[2]{ \left\langle #1 , #2 \right\rangle_{  H^* , H } }

 \newcommand{\HON}[1]{ \| #1 \|_{ H^1} }

\newcommand{ \Om }{ \Omega}

\newcommand{ \pOm}{\partial \Omega}

\newcommand{\D}{ \mathcal{D} \left( \Omega \right)}

\newcommand{\DP}{ \mathcal{D}^{\prime} \left( \Omega \right)  }

\newcommand{\DPP}[2]{   \left\langle #1 , #2 \right\rangle_{  \mathcal{D}^{\prime}, \mathcal{D} }}

\newcommand{\PHH}[2]{    \left\langle #1 , #2 \right\rangle_{    \left(H^1 \right)^*  ,  H^1   }    }

\newcommand{\PHO}[2]{  \left\langle #1 , #2 \right\rangle_{  H^{-1}  , H_0^1  }}

 \newcommand{\HO}{ H^1 \left( \Omega \right)}

\newcommand{\HOO}{ H_0^1 \left( \Omega \right) }

\newcommand{\CC}{C_c^\infty\left(\Omega \right) }

\newcommand{\N}[1]{ \left\| #1\right\|_{ H_0^1  }  }

\newcommand{\IN}[2]{ \left(#1,#2\right)_{  H_0^1} }

\newcommand{\INI}[2]{ \left( #1 ,#2 \right)_ { H^1}}

\newcommand{\HH}{   H^1 \left( \Omega \right)^* }

\newcommand{\HL}{ H^{-1} \left( \Omega \right) }

\newcommand{\HS}[1]{ \| #1 \|_{H^*}}

\newcommand{\HSI}[2]{ \left( #1 , #2 \right)_{ H^*}}

\newcommand{\WO}{ W_0^{1,p}}
\newcommand{\w}[1]{ \| #1 \|_{W_0^{1,p}}}

\newcommand{\ww}{(W_0^{1,p})^*}

\newcommand{\Ov}{ \overline{\Omega}}

\title{A nonlinear  elliptic problem involving the gradient on a half space}
\maketitle

\begin{abstract}
We consider perturbations of the  diffusive Hamilton-Jacobi equation
\begin{equation*} 
 \left\{ \begin{array}{lcl}
\hfill  -\Delta u  &=&  (1+g(x))| \nabla u|^p\qquad \mbox{ in } \IR^N_+,   \\
\hfill  u &=& 0 \hfill \mbox{ on }   \partial \IR^N_+,
\end{array}\right.
  \end{equation*}
  for $ p>1$.  We prove the existence of a classical solution provided $ p \in (\frac{4}{3},2)$ and $g$ is bounded with uniform radial decay to zero.

\end{abstract}


\section{Introduction}

In this work we will investigate perturbations of
\begin{equation} \label{non_pert}
 \left\{ \begin{array}{lcl}
\hfill  -\Delta u  &=&  | \nabla u|^p\qquad \mbox{ in } \IR^N_+,   \\
\hfill  u &=& 0 \hfill \mbox{ on }   \partial \IR^N_+,
\end{array}\right.
  \end{equation}
where $\IR_+^N=\{(x_1,...,x_N)\in\IR^N,~~x_N>0\}$  and $\frac{4}{3}<p<2$.  In particular we are interested in classical nonzero solutions.

\begin{exam}  For $ t>0$ set
\[ u_t(x):=\int_0^{x_N} \frac{1}{ \left( (p-1) y + t \right)^\frac{1}{p-1}} dy. \]   A computation shows that for $ p>1$, $u_t$ is a classical solution of $(\ref{non_pert})$. For $p>2$ the solution is unbounded when $ x_N \rightarrow \infty$ and when $ 1<p<2$ the solution is bounded.  Note that this solution has a closed form. Also note that $u_t$ converges to zero as $ t \rightarrow \infty$.
\end{exam}

A particular perturbation of the above problem will be
\begin{equation} \label{pert_ha}
 \left\{ \begin{array}{lcl}
\hfill  -\Delta u  &=& (1+g(x)) | \nabla u|^p\qquad \mbox{ in } \IR^N_+,   \\
\hfill  u &=& 0 \hfill \mbox{ on }   \partial \IR^N_+.
\end{array}\right.
  \end{equation} In particular we are interested in nonzero solutions of (\ref{pert_ha}) for sufficient smooth functions $g$ which satisfy needed assumptions.  Our approach will be to linearize around $u_t$ to obtain solutions of (\ref{pert_ha}).

  We now state our main theorem.

\begin{thm} \label{main_t}Suppose $ \frac{4}{3}<p<2$ and  $ g$ is bounded, H\"older continuous and satisfies
\[ \sup_{|x|>R, \ x_N \ge 0} | g(x)| \rightarrow 0 \qquad \mbox{ as $ R \rightarrow \infty$.}\]  Then there is a nonzero classical solution of $(\ref{pert_ha})$.
\end{thm}

  \begin{remark} \begin{enumerate}  \item  The conditions on $g$ can surely be  weakened but our interest was mainly in not making any smallness assumptions on $ g$.

  \item  The condition on $p$ may seem somewhat arbitrary but we mention that  the restriction  $ \frac{4}{3}<p<2$ ensures that $  \mu + 1 - \alpha \in (0,1)$ $($see Section $\ref{norms and}$ $)$  which is needed for the proof of   Liouville-type theorems,  Propositions $\ref{pp3}$ and $\ref{pp4}$, that arose in the blow up analysis.
  \end{enumerate}
  \end{remark}

\subsection{Background}
A well studied problem is the existence versus non-existence of positive solutions of the Lane-Emden equation given by
\begin{eqnarray} \label{lane}
 \left\{ \begin{array}{lcl}
\hfill   -\Delta u    &=& u^p  \qquad \mbox{ in }  \Omega,  \\
\hfill u &=& 0 \hfill \mbox{ on }  \pOm,
\end{array}\right.
  \end{eqnarray} where $1<p$ and $ \Omega$ is a bounded domain in $ \IR^N$ (where $N \ge 3$) with smooth boundary.   In the subcritical case $ 1<p<\frac{N+2}{N-2}$  the problem is very well understood and $H_0^1(\Omega)$ solutions are classical solutions; see \cite{gidas}.  In the case of $ p \ge \frac{N+2}{N-2}$  there are no classical positive solutions in the case of the domain being star-shaped; see \cite{POHO}.  In the case of non star-shaped domains much less is known; see for instance \cite{Coron, M_1, M_2, M_3, Passaseo}.   In the case of $ 1<p<\frac{N}{N-2}$ ultra weak solutions (non $H_0^1$ solutions) can be shown to be classical solutions.  For $ \frac{N}{N-2}<p< \frac{N+2}{N-2}$  one cannot use elliptic regularity to show ultra weak solutions are classical.  In particular in \cite{MP} for a general bounded domain in $ \IR^N$  they construct singular ultra weak solutions with a prescribed singular set, see the book  \cite{Pacard} for more details on this.

We now consider
\begin{eqnarray} \label{lane_grad}
 \left\{ \begin{array}{lcl}
\hfill   -\Delta u    &=& |\nabla u|^p  \qquad \mbox{ in }  \Omega,  \\
\hfill u &=& 0 \hfill \mbox{ on }  \pOm,
\end{array}\right.
  \end{eqnarray} where $\Omega$ is a bounded domain in $ \IR^N$.
  The first point is that it is a non variational equation and hence there are various standard tools which are not available anymore.  The case $0<p<1$ has been studied in \cite{ACL}.  Some relevant monographs for this work include \cite{GT,GhRa2,STRUWE}.
Many people have studied boundary blow up  versions of (\ref{lane_grad}) where one removes the minus sign in front of the Laplacian;  see for instance  \cite{LaLi,ZZha}.
 See \cite{ACL, ABTP,ACTMOP,ACJT,ACJT2,BBM,BHV,BHV2, BHV3,ChC,FeMu,FePoRa,GPS,GPS2,GMP,GrTr,PS,Lions,MN,Ng,NV} for more results on equations similar to (\ref{lane_grad}). In particular, the interested reader is referred to \cite{Ng} for recent developments and a bibliography of significant earlier work, where the author studies isolated
singularities at $0$ of nonnegative solutions of the more general quasilinear equation
\[\Delta u=|x|^\alpha u^p+|x|^\beta |\nabla u|^q~~in~~\Omega\setminus\{0\},\]
where $\Omega\subset R^N$ ($N>2$) is a $C^2$ bounded domain containing the origin $0$, $\alpha>-2$, $\beta>-1$ and $p,q>1$, and provides a full classification of positive solutions  vanishing on $\partial \Omega$ and the removability of isolated singularities.\\
Let us finally mention that for   the whole space case, it was proved in \cite{Lions} that any classical solution of (\ref{lane_grad}) when  $\Omega=\IR^N$ with $p > 1$ has to be constant. Also, for the half-space problem (\ref{non_pert}) in the superquadratic case $p>2$, it was proved in \cite{souplet_half}  a Liouville-type classification, or symmetry result, which asserts that any solution $u\in C^2(\IR^N_+)\cap C(\overline{\IR^N_+)}$
has to be one-dimensional, where the result was obtained by using moving planes technique, combined with Bernstein type estimates  and a compactness argument. A similar result in the subquadratic case $p\in(1,2]$   was proved in \cite{PV}.
\\

Before outlining our approach we mention that our work is heavily inspired by the works
\cite{pert_ori,MP,Pacard,davila,davila_fast,pert_wei, small_hole}.  Many of these works consider variations of  $-\Delta u = u^p$ on the full space or an exterior domain.  Their approach is to find an approximate solution and then to linearize around the approximate solution to find a true solution.   This generally involves a very detailed linear analysis of the linearized operator associated with approximate solution and then one applies a fixed point argument to find a true solution.

This current work continues the  theme   of examining $ -\Delta u = | \nabla u|^p$ (or variations) for singular or classical solutions, see \cite{Cowanr1, Cowanr, ACL2018, Calc_var, JDE_2}.

We also mention the recent work \cite{souplet_half} where they examine various results, some of which are Liouville theorems related to (\ref{non_pert}).

  \subsection{Outline of approach}

  First we note that by a scaling argument, instead of finding a nonzero  solution
 of  (\ref{pert_ha}), it is sufficient to find a nonzero solution of
 \begin{equation} \label{pert_ha_1}
 \left\{ \begin{array}{lcl}
\hfill  -\Delta u(x)  &=& (1+g(\lambda x)) | \nabla u(x)|^p\qquad \mbox{ in } \IR^N_+,   \\
\hfill  u &=& 0 \hfill \mbox{ on }   \partial \IR^N_+,
\end{array}\right.
  \end{equation} for some $ \lambda>0$.   We will look for a solution of (\ref{pert_ha_1}) of the form $ u(x)= u_t(x) + \phi(x)$ (where $ t=1$;  but we leave $ t>0$ arbitrary for now) where $ \phi$ is unkown.
   Then $ \phi$ must satisfy
  \begin{equation} \label{outline_eq}
 \left\{ \begin{array}{lcl}
\hfill  \widetilde{L_t}(\phi)  &=& g(\lambda x) | \nabla u_t + \nabla \phi|^p
+ | \nabla u_t + \nabla \phi|^p \\&&- | \nabla u_t|^p -p | \nabla u_t |^{p-2} \nabla u_t \cdot \nabla \phi\qquad \mbox{ in } \IR^N_+,   \\
\hfill  \phi &=& 0 \hfill \mbox{ on }   \partial \IR^N_+,
\end{array}\right.
  \end{equation} where the arguments for all the functions are $x$ except for $g$ and where a computation shows that
\[ \widetilde{L_t}(\phi):=-\Delta \phi - p | \nabla u_t|^{p-2} \nabla u_t \cdot \nabla \phi = -\Delta \phi - \frac{p \phi_{x_N}}{ (p-1) x_N +t}.\]  We will develop a linear theory for the mapping $L_t$, a rescaled version of $ \widetilde{L_t}$.  We will show for all $ t>0$ there is some $C_t>0$ such  that for all $ f \in Y$ there is some $ \phi \in X$ (see Section \ref{norms and} for the definition $X$ and $Y$) which satisfies $L_t(\phi)=f$ in $ \IR^N_+$ with $ \phi=0$ on $ \partial \IR^N_+$.  Moreover one has $ \| \phi\|_X \le C_t \|f\|_Y$.   Using this we will find a solution of (\ref{outline_eq}) using a fixed point argument.  Toward this we define a nonlinear mapping on  $B_R$ (the closed ball of radius $R$ centered at the origin in $X$) by $J_\lambda(\phi)=J_{\lambda,t}(\phi)=\psi$, where
 \begin{equation} \label{nonlinear_map}
 \left\{ \begin{array}{lcl}
\hfill  \widetilde{L_t}(\psi)  &=& g(\lambda x) | \nabla u_t + \nabla \phi|^p  + | \nabla u_t + \nabla \phi|^p \\&& - | \nabla u_t|^p -p | \nabla u_t |^{p-2} \nabla u_t \cdot \nabla \phi\qquad \mbox{ in } \IR^N_+,   \\
\hfill  \psi &=& 0 \hfill \mbox{ on }   \partial \IR^N_+.
\end{array}\right.
  \end{equation}

\section{The linear theory}

We begin by collecting the various parameters and function spaces for the reader's convenience.  \\

\subsubsection{The parameters, spaces and linear operators.}\label{norms and}   Let $p \in (\frac{4}{3},2)$, $ \alpha = \frac{1}{p-1} >1$, $ \gamma= \frac{p}{p-1}>1$, $\mu = \frac{\gamma}{2}$ (note this implies that $ \mu + 1 - \alpha \in (0,1)$) and $ \sigma>0$ small (chosen small enough so that our solution in the end is a classical solution after applying elliptic regularity). We introduce the norms
\[ \| \phi \|_X:=\sup_{0<x_N\le 1} |x_N|^\sigma | \nabla \phi(x)| + \sup_{x_N \ge 1} |x_N|^\alpha | \nabla \phi(x)|,\]
\[ \|f\|_Y:=\sup_{0<x_N \le 1} |x_N|^{\sigma+1} | f(x)| + \sup_{x_N \ge 1} |x_N|^{\alpha+1} |f(x)|,\] where for $\phi \in X$ we require $\phi=0$ on $ \partial \IR^N_+$.   The linear operator we deal with is, for $ t \ge 1$,
\[ L_t(\phi)=\Delta \phi + \frac{\gamma \phi_{x_N}}{x_N+t},\] and note that
\[ \widetilde{L_t}(\phi)= - L_\frac{t}{p-1}(\phi).\]

After considering the operator $L_t$ it is natural to consider a slight modification of the space $X$ (call it $ \widehat{X}$) whose norm is given by
\begin{eqnarray*}
\| \phi \|_{\widehat{X}}&:=&\sup_{0<x_N\le 1} \left\{ |x_N|^\sigma | \nabla \phi(x)| + |x_N|^{\sigma+1} | \Delta \phi(x)| \right\} \\
&&+ \sup_{x_N \ge 1} \left\{  |x_N|^\alpha | \nabla \phi(x)| + |x_N|^{\alpha+1} | \Delta \phi(x)| \right\},
\end{eqnarray*} so we are defining $\widehat{X}:=\{ \phi : \| \phi\|_{\widehat{X}}<\infty \mbox{ and }   \phi =0 \mbox{ on } \partial \IR^N_+\}$.

We will use a change of variables  $ \psi(x)=(x_N+t)^\mu \phi(x)$ and set $L^t$ by
\[ L^t(\psi):=-\Delta \psi+ \frac{\mu(\mu-1) \psi}{(x_N+t)^2}.\] Then $ L_t(\phi)=f$ in $ \IR^N_+$  if $L^t(\psi)= (x_N+t)^\mu f(x)$ in $ \IR^N_+$.  The natural function spaces for $\psi$ are endowed with the norms
 \[ \| \psi \|_{X_\psi}:= \sup_{0<x_N<1} |x_N|^{\sigma-1} | \psi(x)| + \sup_{x_N>1} |x_N|^{\alpha-1-\mu} | \psi(x)|,\] where as before we take $ \psi=0$ on $ \partial \IR^N_+$,  the $Y_\psi$ norm is given by
\[ \|h\|_{Y_\psi}:=\sup_{0<x_N<1} |x_N|^{\sigma+1} |h(x)|+ \sup_{x_N>1} |x_N|^{\alpha+1-\mu}|h(x)|.\]  Again it is natural to consider the modified $ X_\psi$ norm given by
\begin{eqnarray*}
 \| \psi \|_{\widehat{X_\psi}}&:=& \sup_{0<x_N<1} \left\{|x_N|^{\sigma-1} | \psi(x)|+|x_N|^\sigma | \nabla \psi(x)| + |x_N|^{\sigma+1} | \Delta \psi(x)| \right\} \\
 &&+ \sup_{x_N>1} \left\{  |x_N|^{\alpha-1-\mu} | \psi(x)| + |x_N|^{\alpha-\mu} | \nabla \psi(x)| + |x_N|^{\alpha+1-\mu} | \Delta \psi(x)| \right\},
 \end{eqnarray*}  where we are imposing $ \psi=0$ on $ \partial \IR^N_+$.

\subsection{The linear theory}

We need to consider the following equation
\begin{equation} \label{linear_main}
 \left\{ \begin{array}{lcl}
\hfill L_t( \phi) &=&  f(x) \quad \mbox{ in } \IR^N_+,   \\
\hfill  \phi &=& 0 \hfill \mbox{ on  } \partial \IR^N_+. \\
\end{array}\right.
  \end{equation}

\begin{thm}\label{main_lin_theor} For all $ t  \ge 1$ there is some $ C=C_t$ such that for all $ f \in Y$ there is some $ \phi \in X$ which satisfies $(\ref{linear_main})$ and $ \| \phi \|_X \le C \|f\|_Y$.
\end{thm} Instead of working directly with $ \phi$ we prefer to use a change of variables.  If we set  $ \psi(x)=(x_N+t)^\mu \phi(x)$ and set $L^t$ by
\[ L^t(\psi):=-\Delta \psi+ \frac{\mu(\mu-1) \psi}{(x_N+t)^2},\] then it is sufficient to develop a theory for
\begin{equation} \label{linear_main_psi}
 \left\{ \begin{array}{lcl}
\hfill L^t( \psi) &=&  h(x) \quad \mbox{ in } \IR^N_+,   \\
\hfill  \psi &=& 0 \hfill \mbox{ on  } \partial \IR^N_+. \\
\end{array}\right.
  \end{equation} A computation shows that if $\psi$ satisfies (\ref{linear_main_psi}) with $ h(x)=h_f(x)=-(x_N+t)^\mu f(x)$  then $ \phi$ satisfies (\ref{linear_main}).    The result relating the two problems is given by

\begin{prop} Suppose there is some $ C>0$ such that for all $ h \in Y_\psi$ there is some $ \psi \in \widehat{X_\psi}$ that solves $(\ref{linear_main_psi})$ and
 $ \| \psi \|_{\widehat{X_\psi}} \le C \|h \|_{Y_\psi}$.  If we set
$ \phi:= (x_N+t)^{-\mu} \psi$ and put $ h(x)=h_f(x)= -(x_N+t)^\mu f(x)$,  where $ f \in Y$ with $ \|f\|_{Y}=1$,  then $ \phi$ satisfies $(\ref{linear_main})$ and
$ \| \phi \|_{X} \le C_t$.
\end{prop}

\begin{proof} Let $ f \in Y$ with $ \|f\|_Y=1$ and set $ h(x)=h_f(x)=-(x_N+t)^\mu f(x)$.  Then there is some $ C_t$ such that $ \|h\|_{Y_{\psi}} \le C_t$ and hence there is some $ C_{1,t}$ and $ \psi \in \widehat{X_\psi}$ which solves (\ref{linear_main_psi}) and $ \| \psi \|_{\widehat{X_\psi}} \le C_{1,t}$.
A direct computation shows that $ \phi$ satisfies the needed equation. Also   note that
\[ \nabla \phi(x)= \frac{ \nabla \psi(x)}{(x_N+t)^\mu} - \frac{ \mu e_N \psi(x)}{(x_N+t)^{\mu+1}},\] where $e_N$ is the $N^{th}$ coordinate vector.  Since $ \psi \in \widehat{X_\psi}$ one easily sees that $ \phi \in X$ and there is some $C_1$ depending only on $t,p,N$ such that $ \| \phi \|_X \le C_1 \|\psi \|_{\widehat{X_\psi}}$.  This gives the desired result.
\end{proof}

To prove the needed linear theory for $L^t$ we will use a continuation argument and to start the process we will need some results for  Laplacian.

\begin{prop} \label{lap_hom} Assuming the earlier assumptions on the  parameters we have $ \Delta:\widehat{X_\psi} \rightarrow Y_\psi$ is a homomorphism.
\end{prop}

\begin{proof}  \textbf{Into.} Let $ \psi \in \widehat{X_\psi}$ with $ \Delta \psi=0$ in $ \IR^N_+$.    Note that for $ 0<x_N<1$ we have $ | \psi(x)| \le C x_N^{1-\sigma}$ and so $ \psi=0$ on $ \partial \IR^N_+$.  Let $ 1 \le i \le N-1$ and for any fixed $ h\in\IR\setminus \{0\}$ set
\[  \psi^h(x) = \frac{ \psi(x+ h e_i)- \psi(x)}{h},\] and note that $ \psi^h$ is also harmonic in $ \IR^N_+$.  Note also that there is some $ C_h$ such that $ | \psi^h(x)| \le C_h x_N^{1-\sigma}$ for $ 0<x_N<1$.   Also  for $x_N>1$ we have
\[ | \psi^h(x)| \le \int_0^1 | \nabla \psi(x+t h e_i)| dt \le C |x_N|^{\mu-\alpha},\] where $ C$ is independent of $h$ and also  note the exponent $\mu-\alpha$ is negative since $ p<2$. We can extend $ \psi^h$ oddly across $x_N=0$ to see that the extension is harmonic and bounded on $ \IR^N$ and hence is constant.  Taking into account the boundary condition of $ \psi^h$ we see $ \psi^h=0$ and hence $ \psi(x)=\psi(x_N)$ and recalling $ \psi$ is harmonic and the bound near $x_N=0$ we see that $ \psi(x_N)=Ax_N$.  Now recalling for $ x_N>1$ we have $ | \psi(x_N)| \le C x_N^{\mu+1-\alpha}$ and since this exponent is in $(0,1)$ we get $ \psi=0$.

\noindent
\textbf{Onto.} We will find a supersolution on finite domains and then pass to the limit.  To construct our supersolution we will first consider a one dimensional problem.   Firstly consider the one dimensional analogs of the $ X_\psi$ and $ Y_\psi$ norms (written $ X_\psi^1,  Y_\psi^1$) on $ (0,\infty)$.   For $ \tilde{h} \in Y_\psi^1$ we want to find an $\tilde{H} \in X_\psi^1$ which solves
\begin{equation} \label{one_dim}
- \tilde{H}''(x_N)= \tilde{h}(x_N) \quad \mbox{ for } x_N \in (0,\infty), \qquad \mbox{ with } \tilde{H}(0)=0.
\end{equation}
A direct computation shows that
\[ \tilde{H}(x_N)= \int_0^{x_N} \tau \tilde{h}(\tau) d\tau - x_N \int_{\infty}^{x_N} \tilde{h}(\tau) d\tau,\]  satisfies (\ref{one_dim}). Additionally one sees there is some $C$ such that $\|\tilde{H}\|_{X_\psi^1} \le C \|\tilde{h} \|_{Y_\psi^1}$.  Set
\[ \tilde{h}_0(x_N) = \frac{ \chi_{(0,2)}(x_N)}{ x_N^{\sigma+1}} + \frac{ \chi_{(1,\infty)}(x_N)}{ x_N^{\alpha+1-\mu}} \] and let $\tilde{H}_0$ denote the corresponding solution as defined above and set $ \overline{\psi}(x)=\tilde{H}_0(x_N)$; this will be our supersolution on a truncated domain. Now let $ h \in Y_\psi$ with $ \| h\|_{Y_\psi}=1$ and for $ R>1$ (big) and $ \E>0$ (small)
consider $ Q_{R,\E}:=B_R \times (\E,R) \subset \IR^{N-1} \times \IR$.   Let $ C$ be from the 1 dimensional problem.   Let $ \psi_{R,\E}$ denote a solution of
\[ -\Delta \psi_{R,\E}(x) = h(x) \qquad in~~Q_{R,\E} \qquad \psi_{R,\E}=0 \quad on~~\partial Q_{R,\E}.\]   Then by comparison principle we have
$ \overline{\psi}(x) \ge \psi_{R,\E}(x)$ in $Q_{R,\E}$ and one can argue similarly to get $ |\psi_{R,\E}(x)| \le \overline{\psi}(x)  $ in $Q_{R,\E}$.  Hence there is some $C_1>0$ such that for all $ R>1$ and $ 0<\E$ small (and independent of $ h$)  we have
\[ \sup_{0<x_N<1; x \in Q_{R,\E}} x_N^{\sigma-1} | \psi_{R,\E}(x)| + \sup_{x_N>1, x \in Q_{R,\E}} x_N^{\alpha-1-\mu} | \psi_{R,\E}(x)| \le C_1.\]  By taking $ \E= \frac{1}{R}$ and using a diagonal argument and compactness we see that we can pass to the limit to find some $ \psi$ such that $ -\Delta \psi(x)=h(x)$ in $ \IR^N_+$.  Also by fixing $x$ we can pass to the limit in the quantities in the norm and see that $ \psi \in X_\psi$ (hence $ \psi=0$ on $ \partial \IR^N_+$).   Additionally we have $ \| \psi \|_{X_\psi} \le C_1$.  A standard argument now gives the desired bound in $ \widehat{X_\psi}$; we will include the argument for the sake of the reader.  \\
For $ 0<x_N<1$ consider $ \tilde{\psi}(y):= x_N^{-1+\sigma} \psi(x+x_N y)$ for $ y \in B_\frac{1}{4}$.   Fix $ N<q<\infty$ and then by local regularity there is some $ C=C(q,N)$ such that
\begin{equation} \label{loc_reg}
\| \tilde{\psi} \|_{W^{2,q}(B_\frac{1}{8})} \le C \| \Delta \tilde{\psi} \|_{L^q(B_\frac{1}{4})} + C \| \tilde{\psi}\|_{L^q(B_\frac{1}{4})},
\end{equation} and note the bounds on $ h$ and $ \psi$ show that the norms on the right are bounded (independent of $ x$ in the allowable range).  One can now use the Sobolev imbedding to see that
\[ \sup_{B_\frac{1}{8}} | \nabla \tilde{\psi}| \le C_q \| \tilde{\psi} \|_{W^{2,q}(B_\frac{1}{8})}\] and hence we have the gradient bounded;  writing this out in terms of $ \psi$ gives the desired bound on the gradient of $ \psi$.   To get the second order bound we directly use the equation for $ \psi$.

A similar argument gives the desired estimate for $ x_N>1$. Combining these results gives the desired $ \widehat{X_\psi}$ bounds.
\end{proof}

\begin{thm} For all $ t \ge 1$ there is some $ C_t$ such that for all $ h \in Y_\psi$ there is some $ \psi \in \widehat{X_\psi}$ such that $(\ref{linear_main_psi})$ holds and $ \| \psi \|_{\widehat{X_\psi}} \le C_t \|h\|_{Y_\psi}$.
\end{thm}

\begin{proof} Since $ \Delta: \widehat{X_\psi} \rightarrow Y$ is a homomorphism we can use a continuation argument to get the desired result.   So towards this we consider
\[ L^t_\tau(\psi):=-\Delta \psi + \frac{ \tau \mu (\mu-1) \psi}{(x_N+t)^2}.\]  Then note that $ (\tau, \psi) \mapsto L^t_\tau(\psi)$ is a continuous mapping from $ [0,1] \times \widehat{X_\psi}$ to $Y$.  So to get the desired result it is sufficient to get estimates on this mapping uniformly in $\tau$.   So we suppose the result is false and hence there are sequences $ \tau_m \in (0,1]$,  $ \psi_m \in \widehat{X_\psi}$ and $ h_m \in Y_\psi$ such that $ \| \psi_m\|_{\widehat{X_\psi}}=1$ and $ \|h_m \|_{Y_\psi} \rightarrow 0$ and $L^t_{\tau_m}(\psi_m)=h_m$ in $ \IR^N_+$.    We first assume that the zero order term in the norm of $ \psi_m$ is bounded away from zero; so after renormalizing we can assume that  $ \| \psi_m \|_{X_\psi}=1$ and we still have $ \|h_m \|_{Y_\psi} \rightarrow 0$.  For ease of notation now we will slightly switch notation; we will write $ (x,y) \in \IR^{N-1} \times (0,\infty)$ instead of $ x  \in \IR^N_+$.  \\

\noindent
We consider three cases: \\
\noindent
(i) there is $y^m \rightarrow 0 $ such that $ (y^m)^{\sigma-1} | \psi_m(x^m,y^m)| \ge \frac{1}{2}$,\vspace{0.05 in} \\
(ii) there is some $y^m \rightarrow \infty$ such that $ (y^m)^{\alpha-1-\mu} | \psi_m(x^m,y^m)| \ge \frac{1}{2}$, \vspace{0.05 in} \\
(iii) there is some $y^m$ bounded and bounded away from zero such that $ | \psi_m(x^m,y^m)|$ is bounded away from zero. \\
In all three cases we write $ \overline{x^m}=(x^m,y^m)$.  \\

\noindent
\textbf{Case (i).} Set $ \psi^m(z) = (y^m)^{\sigma-1} \psi_m( \overline{x^m}+ y^m z)$ for $ z_N>-1$.  Then $ | \psi^m(0)|$ is bounded away from zero and
\[ | \psi^m(z)| \le (1+z_N)^{1-\sigma} \qquad \mbox{ for } 0<y^m(1+z_N)<1,\] and a computation shows that
\[ -\Delta \psi^m(z) + \frac{\tau_m \mu (\mu-1) \psi^m(z)}{( z_N +1 + (y^m)^{-1} t)^2} = \widehat{h}_m(z) \quad \mbox{ in $z_N>-1$},\] with $ \psi^m=0$ on $z_N=-1$
where $ \widehat{h}^m(z)= (y^m)^{\sigma+1} h_m(\overline{x^m}+y^m z)$.  Note that
\[ |\widehat{h}^m(z)| \le \frac{ \|h_m\|_{Y_\psi}}{(1+z_N)^{\sigma+1}} \qquad 0< y^m(1+z_N)<1,\]
and hence $\widehat{h}_m \rightarrow 0$ uniformly away from $ z_N=-1$.  By a standard compactness and diagonal argument (and after passing to a subsequence) $ \psi^m \rightarrow \psi$  locally in $ C^{1.\delta}_{loc}(z_N>-1)$ and $ \psi$ satisifes $ \Delta \psi(z)=0$ in $ z_N>-1$, $ | \psi(0)| \neq 0$, $ | \psi(z)| \le (1+z_N)^{1-\sigma}$.  Using a similiar argument as in the proof of the previous proposition we see that we must have $ \psi=0$ which is a contradiction. \\

\noindent
\textbf{Case (ii).}  Set $ \psi^m(z) = (y^m)^{\alpha-1-\mu} \psi_m( \overline{x^m}+ y^m z)$ for $ z_N>-1$.  Then $ | \psi^m(0)|$ is bounded away from zero and
\[ | \psi^m(z)| \le (1+z_N)^{\mu+1-\alpha} \quad \mbox{ for } y^m(1+z_N)>1,\] and recall that $ \mu+1-\alpha \in (0,1)$.  One should note there is an estimate valid for $ z_N$ near $ -1$  but we won't need this.   A computation shows that

\[ -\Delta \psi^m(z) + \frac{\tau_m \mu (\mu-1) \psi^m(z)}{( z_N +1 + (y^m)^{-1} t)^2} = \widehat{h}_m(z) \quad \mbox{ in $z_N>-1$},\] with $ \psi^m=0$ on $z_N=-1$, where
 $ \widehat{h}_m(z)= (y^m)^{\alpha-\mu+1} h_m( \overline{x^m} + y^m z)$.  A computation shows that \[ | \widehat{h}_m(z)| \le \frac{\| h_m \|_{Y_\psi}}{(1+z_N)^{\alpha-\mu+1}}, \quad \mbox{ for } y^m(1+z_N)>1,\] and hence $ \widehat{h}_m \rightarrow 0$ uniformly away from $ z_N=-1$. Again by compactness and a diagonal argument we can assume $ \psi^m \rightarrow \psi$ in $ C^{1,\delta}_{loc}(z_N>-1)$ and $ \tau_m \rightarrow \tau \in [0,1]$ and $ \psi$ satisfies
\begin{equation} \label{needed_lio_1}
-\Delta \psi(z) + \frac{\tau \mu (\mu-1) \psi(z)}{( z_N +1 )^2} = 0 \quad \mbox{ in $z_N>-1$},
\end{equation} with $ | \psi(z)| \le (1+z_N)^{\mu+1-\alpha}$ for $ z_N>-1$ and hence $ \psi=0$ on $ z_N=-1$.  We can now apply Proposition \ref{pp4} to get the desired contradiction.
\\

\noindent
\textbf{Case (iii).}  Here we set $ \psi^m(z)=\psi_m( \overline{x^m}+ y^m z)$ for $ z_N>-1$.  Then $ |\psi^m(0)|$ is bounded away from zero and there is some $C$ (independent of $m$) such that
\begin{equation} \label{double_est}
| \psi^m(z)| \le C \chi_{(-1,1)}(z_N) (1+z_N)^{1-\sigma} + C \chi_{(0,\infty)}(z_N) (1+z_n)^{1+\mu-\alpha},
\end{equation}for $ z_N>-1$.  A computation shows that
\[ -\Delta \psi^m(z) + \frac{ \tau_m \mu(\mu-1) \psi^m(z)}{(1+ z_N + (y^m)^{-1} t)^2} = \widehat{h}_m(z) \quad \mbox {in }  z_N>-1,
\]  where $\widehat{h}_m(z)= (y^m)^2 h_m( \overline{x^m} + y^m z)$ and  $\widehat{h}_m \rightarrow 0$ uniformly away from $ z_N=-1$. Using compactness and a diagonal argument we have $ \psi^m \rightarrow \psi$ in $ C^{1,\delta}_{loc}(z_N>-1)$, hence $ \psi$ satisfies
\[ -\Delta \psi(z) + \frac{\tau \mu(\mu-1) \psi(z)}{(1+z_N + T)^2}=0 \quad \mbox{ in } z_N>-1\] with $ T= \frac{t}{(y^\infty)^2}$, where $ y^m \rightarrow y^\infty \in (0,\infty)$.  Note also that $ | \psi(0) | \neq 0$ and $ \psi$ also satisfies the pointwise bound for $ \psi^m$ given in (\ref{double_est}).  We can now apply Proposition \ref{pp3} to get the desired contradiction.  \\

We have proven the desired estimates on $ \| \psi_m \|_{X_\psi}$, i.e., $ \| \psi_m \|_{X_\psi} \rightarrow 0$.  To see that in fact $ \| \psi_m \|_{\widehat{X_\psi}} \rightarrow 0$ one can now use a standard scaling argument, see the end of the  proof of Proposition  \ref{lap_hom} for an idea of the needed scaling argument.
\end{proof}

\subsection{Liouville theorems}

In this section we prove the needed Liouville theorems that arose in the blow up analysis.

\begin{prop} \label{pp3} Let $ t>0$, $ \tau \in [0,1] $ and  $ \psi \in \widehat{X_\psi}$ be such

\begin{equation} \label{liovile_eq_1}
-\Delta \psi(x) + \frac{\tau \mu(\mu-1) \psi(x)}{(x_N + t)^2}=0 \mbox{ in } \IR^N_+.
\end{equation} Then $ \psi=0$.
\end{prop}

\begin{proof} The case of $ \tau=0$ has already been handled since this is just the Laplacian. Again we will switch notation to $ x= \overline{x}= (x,y)$. For $ 1 \le i \le N-1$ and $ 0<|h| \le 1$ we consider
\[ \psi^h(x,y)=  \frac{ \psi( (x,y)+ h e_i)- \psi( x,y)}{h},\] and note that $ \psi^h$ satisfies the same equation as $ \psi$.  Also note that since $t>0$ the equation has no singularities in it at $y=0$ and hence $ \psi$ is in fact smooth up to the boundary.   Also there is some $C>0$ (independent of $h$) such that $ | \psi^h(x,y)| \le C y^{\mu-\alpha}$ for $ x_N>1$ and note $ \mu - \alpha<0$. Also there is some $C_h$ such that $ | \psi^h(x,y)| \le C_h y^{1-\sigma}$ for $ 0<y<1$ and again we have $ \psi^h$ is in fact smooth.   Using the above bounds we see that $ \psi^h$ is bounded and so if we assume its not identically zero we can then assume (after multiplying by $-1$ if needed) that
$\sup_{\IR^N_+} \psi^h=T \in (0,\infty)$.   If this is attained at some $ (x^0,y^0)$ (with $ y^0 \in (0,\infty)$) we get a contradiction via the maximum principle.  Hence there must be some $ (x^m,y^m)$ such that $ \psi^h(x^m,y^m) \rightarrow T$ and note that we must have $ y^m$ bounded and bounded away from zero after considering the pointwise bound.   For $ z_N>-1$ we set $ \zeta_m(z)= \psi^h((x^m,y^m)+ y^m z)$ and note $ \zeta_m(0) \rightarrow T$ and $ \zeta_m \le T$.  Also  note that
\[ | \zeta_m(z)| \le C (y^m)^{\mu-\alpha} (1+z_N)^{\mu-\alpha} \quad \mbox{ for } y^m (1+z_N)>1, \quad \mbox{ and } \]
\[  | \zeta_m(z)| \le C_h (y^m)^{1-\sigma} (1+z_N)^{1-\sigma} \quad \mbox{ for } 0<y^m (1+z_N)<1.\] By a compactness  and diagonal argument we see there is some $ \zeta$ such that $ \psi_m \rightarrow \zeta$ in $ C^{1,\delta}_{loc}(z_N>-1)$ and $\zeta$ satisfies
\[ -\Delta \zeta(z)+ \frac{\tau \mu (\mu-1) \zeta(z)}{(1+z_N + \frac{t}{y^\infty})^2} =0 \quad \mbox{ in }  z_N>-1,\] where $ y^m \rightarrow y^\infty \in (0,\infty)$ and $\zeta$ satisfies the same pointwise bounds as $ \zeta_m$ and hence $\zeta$ is nonconstant on $ z_N>-1$ and attains its maximum at the origin which contradicts the  maximum principle.  From this we see that $ \psi^h$ is zero and hence $ \psi(x) = \psi(x_N)$.     Returning to the equation for $ \psi$ we see it is now an ode of Euler type and hence has solutions of the form
\[ \psi(x_N)= C_1 (x_N+t)^{\beta_+(\tau)} + C_2 (x_N+t)^{\beta_-(\tau)},\] where
\[ \beta_\pm(\tau) = \frac{1}{2} \pm \frac{  \sqrt{1+4 \tau \mu^2-4 \tau \mu}}{2}.\] A compuation shows that $ \beta_+'(\tau)>0$ for $ \tau \in (0,1)$ and hence for $ \tau \in (0,1]$ one has $ \beta_+(\tau)> \beta_+(0)=1$.   Note that
\[ \alpha-1-\mu+ \beta_+(\tau) > \alpha-1-\mu+ \beta_+(0)= \alpha - \mu>0,\] and hence writing out $ \limsup_{x_N \rightarrow \infty} x_N^{\alpha -1 - \mu} | \psi(x_N)| \le C$  gives that $ C_1=0$.  To satisfy the boundary condition one sees they must have $C_2=0$ and hence $ \psi=0$.
\end{proof}

\begin{prop} \label{pp4} Suppose $ \tau \in [0,1]$ and $ \psi$ satisfies
\begin{equation}
\label{infinity_liouv}
-\Delta \psi(x) + \frac{ \tau \mu(\mu-1) \psi(x)}{x_N^2}=0 \quad \mbox{ in }  \IR^N_+,
\end{equation} with $ | \psi(x)| \le C x_N^{\mu+1-\alpha}$ for $ x \in \IR^N_+$.  Then $ \psi=0$.

\end{prop}
\begin{proof} The case of $ \tau=0$ is handled in the proof of a previous result. As in the previous proof, for $ 1 \le i \le N-1$ and  $ 0<|h| \le 1$, we consider
\[ \psi^h(x,y)=  \frac{ \psi( (x,y)+ h e_i)- \psi(x,y)}{h},\] and note that $ \psi^h$ satisfies the same equation as $ \psi$.  Note this time the equation is singular on the boundary.

Also there is some $C>0$ (independent of $h$) such that $ | \psi^h(x,y)| \le C y^{\mu-\alpha}$ for all $y>0$  and note $ \mu - \alpha<0$. Also there is some $C_h$ such that $ | \psi^h(x,y)| \le C_h y^{\mu-\alpha+1}$ for all $y>0$ and this exponent is positive.     Combining the pointwise estimates we see there is some $ \E>0$ such that
\[ \sup_{(x,y) \in \IR^N_+} | \psi^h(x,y)| = \sup_{ (x,y) \in \IR^{N-1} \times (\E,\E^{-1})} | \psi^h(x,y)|.\]  We can argue exactly as in the previous case to see that $ \psi(x) = \psi(x_N)$ (we have switched notation back to just $ x \in \IR^N_+$).    So we have
\[ \psi(x_N)=C_1 x_N^{\beta_+(\tau)}+ C_2 x_N^{\beta_-(\tau)},\] where $\beta_\pm(\tau)$ is from the previous proof.   Provided we have both $ \beta_+(\tau),\beta_-(\tau)$ different from $ \mu+1-\alpha$  then by sending $ x_N \rightarrow 0,\infty$ we can see $C_1=C_2=0$. From the previous proof we know that $ \beta_+(\tau)>\mu+1-\alpha$ for $ \tau>0$. By using monotonicity in $ \tau$ one sees that $ \beta_-(\tau) < \mu+1-\alpha$ and this gives us the desired result.

\end{proof}

\section{The fixed point argument}

We now will fix $ t=1$.  The following lemma includes some fairly standard inequalities that are needed to prove the nonlinear mapping is a contraction.   Note there are no smallness assumptions on the $y$ and $z$ terms.  See, for instance,  \cite{JDE_2,PLLT} for a proof.

\begin{lemma}  Suppose $ 1<p \le 2$.  Then there is some $ C=C_p$ such that  for all vectors $x,y,z \in \IR^N$ one has
\begin{equation} \label{ineq_1}
0 \le |x+y|^p- |x|^p - p |x|^{p-2} x \cdot y \le C |y|^p,
\end{equation}
\begin{equation} \label{ineq_2}
\Big| |x+y|^p-p |x|^{p-2} x \cdot y - |x+z|^p + p|x|^{p-2} x \cdot z \Big| \le C \left( |y|^{p-1} + |z|^{p-1} \right) |y-z|.
\end{equation}

\begin{equation} \label{ineq_3}
\Big|  |x+y|^p-|x+z|^p \Big| \le C \left( |y|^{p-1} + |z|^{p-1} + |x|^{p-1} \right) |y-z|.
\end{equation}

\end{lemma}

We will now prove Theorem \ref{main_t} and for the readers convenience we restate the theorem.

\begin{thm*} Suppose $ \frac{4}{3}<p<2$ and  $ g$ is bounded, H\"older continuous and satisfies
\[ \sup_{|x|>R, \ x_N \ge 0} | g(x)| \rightarrow 0 \qquad \mbox{ as $ R \rightarrow \infty$.}\]  Then there is a nonzero classical solution of $(\ref{pert_ha})$.
\end{thm*}

\noindent
\textbf{Proof of Theorem \ref{main_t}.}  We will show that $ J_\lambda$ is a contraction mapping on $B_R$ as we outlined in the outline.
In what follows $ C$ is a constant that can change from line to line but is independent of $ \lambda$ and $ R$.\\

\noindent
\textbf{Into.} Let $ 0<R \le 1$,  $ \phi \in B_R$ and let $ \psi = J_\lambda(\phi)$. Then $\psi$ satisfies (\ref{nonlinear_map}) and by the linear theory (Theorem \ref{main_lin_theor}) and using (\ref{ineq_1}) we see that
\begin{eqnarray*}
\| J_\lambda(\phi) \|_X=\| \psi \|_X & \le &C \|g(\lambda x) | \nabla u_t|^p\|_Y + C \|g(\lambda x) | \nabla \phi |^p\|_Y + C \|  | \nabla u_t + \nabla \phi|^p \\&& - | \nabla u_t|^p -p | \nabla u_t |^{p-2} \nabla u_t \cdot \nabla \phi \|_Y \\
& \le & C \|g(\lambda x) | \nabla u_t|^p\|_Y + C \| | \nabla \phi|^p \|_Y
\end{eqnarray*} since $g$ is bounded.  Using the bound on $ \phi$ we see that $ \| | \nabla \phi |^p \|_Y \le C R^p$ and we now examine the other term.  So towards this we set
\[ I^1_\lambda := \sup_{0<x_N<1} x_N^{\sigma+1} |g(\lambda x) | | \nabla u_t|^p \le C \sup_{0<x_N<1} \frac{ x_N^{\sigma+1}}{(x_N+t)^{\alpha p}} |g( \lambda x)|, \quad \mbox{ and }\]

\[ I^2_\lambda := \sup_{x_N>1} x_N^{\alpha+1} |g(\lambda x) | | \nabla u_t|^p \le C \sup_{x_N>1} \frac{ x_N^{\alpha+1}}{(x_N+t)^{\alpha p}} |g( \lambda x)|,\] and note that $ \|g(\lambda x) | \nabla u_t|^p \|_Y \le I_\lambda^1 + I_\lambda^2$.   Let $ 0<\delta<1$ (small). Set $A(T):= \sup_{x_N>0, |x|>T}|g(x)|$ and recall that $A(T) \rightarrow 0$ as $ T \rightarrow \infty$.   Then
\begin{eqnarray*}
I_\lambda^1 &\le & C\sup_{0<x_N<\delta} \frac{ x_N^{\sigma+1}}{(x_N+t)^{\alpha p}} |g( \lambda x)| +C\sup_{\delta<x_N<1} \frac{ x_N^{\sigma+1}}{(x_N+t)^{\alpha p}} |g( \lambda x)| \\
& \le & C \delta^{\sigma+1} +C A(\lambda \delta).
\end{eqnarray*} Similarly one sees that $I_\lambda^2 \le C A(\lambda)$.  Combining the above results and using the fact that $ A$ is monotonic we see that
\[ \|J_\lambda(\phi) \|_X \le C \left\{ R^p + \delta^{\sigma+1} +A(\lambda \delta) \right\}.\] So for $J_\lambda(B_R) \subset B_R$ it is sufficient that
\begin{equation} \label{into_con}
C \left\{ R^p + \delta^{\sigma+1} +A(\lambda \delta) \right\} \le R.
\end{equation} \\

\noindent
\textbf{Contraction.} Let $ 0<R \le 1$,  $ \phi_i \in B_R$ and $ \psi_i = J_\lambda(\phi_i)$, $i=1,2$.  Writing out the equations for $ \psi_2$ and $ \psi_1$ and taking a difference and using (\ref{ineq_2}) and (\ref{ineq_3}) we arrive at
\[\|  J_\lambda(\phi_2) -  J_\lambda(\phi_1)|_X= \| \psi_2 - \psi_1\|_X \le C H_\lambda + C K_\lambda,\] where
\[ H_\lambda = \big\| g(\lambda x) \left\{ | \nabla u_t|^{p-1} + | \nabla \phi_2|^{p-1} + | \nabla \phi_1|^{p-1} \right\} | \nabla \phi_2 - \nabla \phi_1| \big\|_Y, \qquad \mbox{ and }\]

\[ K_\lambda = \big\|\left\{ | \nabla \phi_2 |^{p-1} + | \nabla \phi_1|^{p-1}  \right\} | \nabla \phi_2 - \nabla \phi_1| \|_Y.\]  We first estimate $K_\lambda$.  So using the bound on $ \phi_2$ one can see
\begin{eqnarray*}
\sup_{0<x_N<1} x_N^{\sigma+1} | \nabla \phi_2|^{p-1} | \nabla \phi_2 - \nabla \phi_1| & \le & R^{p-1}\sup_{0<x_N<1} x_N^{\sigma+1-\sigma(p-1)-\sigma} x_N^\sigma | \nabla \phi_2 - \nabla \phi_1| \\
& \le & R^{p-1} \| \phi_2 - \phi_1 \|_X
\end{eqnarray*}
provided $ \sigma+1-\sigma(p-1)-\sigma \ge 0$, which is satisfied after recalling we are taking $ \sigma>0$ very small. A similar argument shows that
\[ \sup_{x_N>1} x_N^{\alpha+1} | \nabla \phi_2|^{p-1} | \nabla \phi_2 - \nabla \phi_1| \le R^{p-1} \| \phi_2 - \phi_1\|_X \sup_{x_N>1} x_N^{\alpha+1-\alpha - \alpha(p-1)},\] and so here we need the exponent to be less or equal zero and note that this exponent is zero.  Combining these two results we see that $K_\lambda \le C R^{p-1} \| \phi_2 - \phi_1\|_X$.

We now examine the $H_\lambda$ term.

\begin{itemize}
    \item First we examine the term $ \| g(\lambda x) | \nabla u_t|^{p-1} | \nabla \phi_2 - \nabla \phi_1| \|_Y.$  Using an argument as before one has
    \[ \sup_{0<x_N<1} x_N |g(\lambda x)| \le C \delta + C A(\lambda \delta).\] A computation shows that
    \begin{eqnarray*}
    \sup_{0<x_N<1} | g(\lambda x)| | \nabla u_t|^{p-1} | \nabla \phi_2 - \nabla \phi_1| &\le& C \| \phi_2 - \phi_1\|_X \sup_{0<x_N<1}  x_N |g(\lambda x)| \\
    & \le & C (\delta+A(\lambda \delta)) \|\phi_2 - \phi_1 \|_X.
    \end{eqnarray*} We now examine the outer portion of the norm,
    \begin{eqnarray*}
    \sup_{x_N>1} x_N^{\alpha+1} |g(\lambda x)| | \nabla u_t|^{p-1} | \nabla \phi_2 - \nabla \phi_1| & = & C \sup_{x_N>1} \frac{ x_N}{x_N+t} | g(\lambda x)|  \left\{ x_N^\alpha | \nabla \phi_2- \nabla \phi_1| \right\} \\
    & \le & CA(\lambda) \| \phi_2 -  \phi_1\|_X.
    \end{eqnarray*} Combining the results gives
     \begin{equation} \label{one_tt}
     \| g(\lambda x) | \nabla u_t|^{p-1} | \nabla \phi_2 - \nabla \phi_1| \|_Y \le C\left\{ \delta + A(\lambda \delta) \right\} \| \phi_2 - \phi_1\|_X
     \end{equation} after using monotonicity of $A$.

 \item We now examine the term  $ \big\| g(\lambda x)  | \nabla \phi_2|^{p-1}  | \nabla \phi_2 - \nabla \phi_1 | \big\|_Y.$ Using the estimate for $ \phi_2$ one sees that
 \[ \sup_{0<x_N<1} x_N^{\sigma+1} | g(\lambda x)| | \nabla \phi_2|^{p-1} | \nabla \phi_2 - \nabla \phi_1| \le R^{p-1} \| \phi_2 - \phi_1\|_X \sup_{0<x_N<1} x_N^{1-\sigma(p-1)} |g(\lambda x)|.\] A computation as before shows that
 \[ \sup_{0<x_N<1} x_N^{1-\sigma(p-1)} |g(\lambda x)| \le C \delta^{1-\sigma(p-1)} + C A(\lambda \delta),\] and hence
 \[  \sup_{0<x_N<1} x_N^{\sigma+1} | g(\lambda x)| | \nabla \phi_2|^{p-1} | \nabla \phi_2 - \nabla \phi_1| \le C R^{p-1} \left\{\delta^{1-\sigma(p-1)} +  A(\lambda \delta) \right\} \| \phi_2 - \phi_1 \|_X.\] Similarly the outer portion of the norm gives
 \begin{eqnarray*}
 \sup_{x_N>1} x_N^{\alpha+1}|g(\lambda x)| | \nabla \phi_2|^{p-1} | \nabla \phi_2 - \nabla \phi_1| &\le& R^{p-1} \sup_{x_N>1} |g(\lambda x)| x_N^{\alpha} | \nabla \phi_2 - \nabla \phi_1| \\
 &\le & R^{p-1} A(\lambda) \| \phi_2 - \phi_1\|_X,
 \end{eqnarray*}
and hence combining these two results gives
 \begin{equation}
     \label{one_ttt}
 \big\| g(\lambda x)  | \nabla \phi_2|^{p-1}  | \nabla \phi_2 - \nabla \phi_1 \big\|_Y  \le C R^{p-1} \left\{ \delta^{1-\sigma(p-1)} + A(\lambda \delta) \right\} \| \phi_2 - \phi_1\|_X,
     \end{equation} where again we have used the monotonicity of $A$.

   \end{itemize}

     Combining with the previous results gives
     \[ H_\lambda \le C \left\{ \delta + A(\lambda \delta) +R^{p-1} \left( \delta^{1-\sigma(p-1)} +A(\lambda \delta) \right) \right\} \| \phi_2 - \phi_1 \|_X.\]  Combining the estimates for  $H_\lambda$ and $K_\lambda$ shows that

  \[\|  J_\lambda(\phi_2) -  J_\lambda(\phi_1)|_X\le
  C \left\{R^{p-1}+ \delta + A(\lambda \delta) +R^{p-1} \left( \delta^{1-\sigma(p-1)} +A(\lambda \delta) \right) \right\} \| \phi_2 - \phi_1 \|_X.\]
 Hence,  $J_\lambda$ is a contraction on $B_R$ provided
     \begin{equation} \label{contr_need}
     C \left\{ R^{p-1}+\delta + A(\lambda \delta) +R^{p-1} \left( \delta^{1-\sigma(p-1)} +A(\lambda \delta) \right) \right\} \le \frac{3}{4}.
     \end{equation}

    So for $J_\lambda$ to be a self-map and
contraction mapping on $B_R$ we need both (\ref{into_con}) and (\ref{contr_need}) to hold.   To pick the $ R,\delta,\lambda$ one first chooses $ R>0$ very small but fixed, then fixes $ \delta$ very small and finally picks $ \lambda$ very big.
    Once $ J_\lambda$ is a contraction we can use Banach's Contraction Mapping Principle to see there is a fixed point $ \phi \in B_R$ and hence we see that  $ u(x) = u_t(x) + \phi(x)$ is a solution of (\ref{pert_ha_1}).  Note that $u_t$ is smooth and the gradient of $ \phi$ can have slight blow up at $ x_N=0$;  depending on $ \sigma>0$.  By taking $ \sigma>0$ very small one can apply elliptic regularity to see that $u$ is  a classical solution. To see that $u$ is not identically zero one needs to choose $R>0$ sufficiently small (relative to $ \|u_t\|_X$) and then one sees that $| \nabla u(x)|>0$ for $x_N>1$ (for instance).
    \hfill $\Box$

\section*{Acknowledgment} A. Aghajani was partially supported by Grant from IPM (No. 1400350211). C. Cowan and S. H. Lui were partially supported by grants from NSERC.

\end{document}